\documentclass[a4paper,reqno,11pt,oneside]{amsart}
\usepackage{amsmath,geometry}
\usepackage{color}

\numberwithin{equation}{section}

\newtheorem{thm}{Theorem}[section]

\newtheorem{lm}[thm]{Lemma}
\newtheorem{oss}[thm]{Remark}
\newtheorem{cor}[thm]{Corollary}

\title{Harnack inequalities and quantization properties for the $n-$Liouville equation}

\author{Pierpaolo Esposito}
\address{Pierpaolo Esposito, Dipartimento di Matematica e Fisica, Universit\`a degli Studi Roma Tre, Largo S.~Leonardo Murialdo 1, Roma 00146, Italy.}
\email{esposito@mat.uniroma3.it}

\author{Marcello Lucia}
\address{Marcello Lucia, Department of Mathematics, College of Staten Island $\&$ The Graduate Center, The City University of New York, 365 Fifth Ave., New York, NY 10016, USA.}
\email{marcello.lucia@csi.cuny.edu}

\begin{document}

\begin{abstract}
We consider a quasilinear equation involving the $n-$Laplacian and an exponential nonlinearity, a problem that includes the celebrated Liouville equation in the plane as a special case. For a non-compact sequence of solutions it is known that the exponential nonlinearity converges, up to a subsequence, to a sum of Dirac measures. By performing a precise local asymptotic analysis we complete such a result by showing that the corresponding Dirac masses are quantized as multiples of a given one, related to the mass of limiting profiles after rescaling according to the classification result obtained by the first author in \cite{Esp}. A fundamental tool is provided here by some Harnack inequality of ``sup+inf" type, a question of independent interest that we prove in the quasilinear context through a new and simple blow-up approach.
\end{abstract}

\maketitle

\section{Introduction}
In the present paper we are concerned with solutions to 
\begin{equation} \label{Liouville-Quasilinear}
   -\Delta_n u = h(x) e^u  \hbox{ in } \Omega ,
\end{equation}
where $\Omega \subset \mathbb{R}^n$, $n \geq 2$, is a bounded open set  and $\Delta_n u = \hbox{div} ( |\nabla u|^{n-1} \nabla u) $ stands for the $n$-Laplace operator. Solutions are meant in a weak sense and by elliptic estimates \cite{Dib,Ser1,Tol} such solutions are in $C^{1,\alpha}(\Omega)$ for some $\alpha \in (0,1)$. 

\medskip
When $n=2$ problem~\eqref{Liouville-Quasilinear} reduces to the so-called Liouville equation, see \cite{Liou}, that represents the simplest case of ``Gauss curvature equation" on a two-dimensional surface arising in differential geometry.
In the higher dimensional case similar geometrical problems have led to different type of curvature equations. Recently, it has been observed that the $n$-Laplace operator comes into play 
when expressing the Ricci curvature after a conformal change of the metric \cite{MaQing}, leading to another class of curvature equations that are of relevance.  Moreover, the $n-$Liouville equation \eqref{Liouville-Quasilinear} represents a simplified version of a quasilinear fourth-order problem arising \cite{EsMa} in the theory of log-determinant functionals, that are relevant in the study of the conformal geometry of a $4-$dimensional closed manifold. In order to understand some of the bubbling phenomena that may occur in such geometrical contexts, we are naturally led to study the simplest situation given by~\eqref{Liouville-Quasilinear}.

\medskip Starting from the seminal work of Brezis and Merle~\cite{BrMerle} in dimension two, the asymptotic behavior of a sequence $u_k$ of solutions to
\begin{equation} \label{758}
-\Delta_n u_k=  h_k (x)  e^{u_k} \hbox{ in }\Omega,
\end{equation} 
with
\begin{equation} \label{758bis}
\sup_k \int_\Omega e^{u_k}<+\infty
\end{equation} 
and $h_k $ in the class
\begin{equation} \label{1059}
\Lambda_{a,b} =\{h  \in C (\Omega): \ a \leq h  \leq b \hbox{ in } \Omega\},
\end{equation}
can be generally described by a ``concentration-compactness" alternative. Extended \cite{AgPe} to the quasi-linear case, it reads as follows.

\medskip

{\bf Concentration-Compactness Principle:}
{\it Consider a sequence of functions $u_k$ such that \eqref{758}-\eqref{758bis} hold with $h_k \in \Lambda_{0,b}$. Then, up to a subsequence, the following alternative holds: 

\begin{enumerate}
\item[(i)] $u_k$ is bounded in $L^{\infty}_{loc} (\Omega)$;
\item[(ii)] $u_k \to - \infty$ locally uniformly in $\Omega$ as $k\to +\infty$;
\item[(iii)] the blow-up set $\mathcal{S}$ of the sequence $u_k$, defined as
$$
   \mathcal{S} = \{ p \in \Omega: \hbox{ there exists }  x_k \in \Omega \hbox{ s.t. } \, x_k \to p , u_k (x_k) \to \infty  \hbox{ as }k \to +\infty \},
$$
is finite, $u_k \to - \infty$ locally uniformly in $\Omega \setminus S$ and  
\begin{equation} \label{253}
h_k e^{u_k} \rightharpoonup \sum_{p \in {\mathcal S}} \beta_p \delta_p
\end{equation}
weakly in the sense of measures as $k \to +\infty$ for some coefficients $\beta_p \geq  n^{n} \omega_n$, 
where $\omega_n$ stands for the volume of the unit ball in $\mathbb R^n$.
\end{enumerate} 
}

\medskip 
The compact case, in which the sequence $e^{u_k}$ does converge locally uniformly in $\Omega$, is expressed by alternatives (i) and (ii), thanks to elliptic estimates \cite{Dib,Tol}; alternative (iii) describes the non-compact case and the characterization of the possible values for the Dirac masses $\beta_p$ becomes crucial towards an accurate description of the blow-up mechanism.

\medskip
When a boundary control on $u_k$ is assumed, the answer is generally very simple. If one assumes that the oscillation of $u_k$ on $\partial B_{\delta} (p)$, $p \in \mathcal S$, is uniformly bounded for some $\delta>0$ small, using a Pohozaev identity it has been shown \cite{EsMo} that  
$\beta_p = c_n \omega_n $, $c_n =n(\frac{n^2}{n-1})^{n-1}$, provided $h_k$ is in the class
\begin{equation} \label{1100}
\Lambda_{a,b}'=\{h  \in C^1 (\Omega): \ a\leq h \leq b,\ |\nabla h|  \leq b  \hbox{ in }\Omega\}
\end{equation}
with $a>0$. Moreover, in the two-dimensional situation and under the condition
\begin{equation} \label{eq:Peso1}
0\leq h_k  \to h \hbox{ in } C_{loc} (\Omega) \hbox{ as } k \to +\infty,
\end{equation}
a general answer has been found by Y.Y.~Li and Shafrir \cite{LiSh} showing that, for any $p \in \mathcal S$, $h(p)>0$ and the concentration mass $\beta_p$ is quantized as follows:
\begin{equation} \label{1011}
\beta_p \in 8\pi \mathbb{N}.
\end{equation}

\medskip \noindent The meaning of the value $8\pi$ in \eqref{1011} can be roughly understood as the sequence $u_k$ was developing several sharp peaks collapsing in $p$, each  of them looking like, after a proper rescaling, as a solution $U$ of 
\begin{equation} \label{Limiting2}
    - \Delta U = h(p) e^U \hbox{ in }\mathbb{R}^2,
    \quad
    \int_{\mathbb R^2} e^U < \infty,
\end{equation}
with $h(p)>0$. Using the complex representation formula obtained by Liouville \cite{Liou} or the more recent PDE approach by Chen-Li~\cite{ChenLi}, the solutions of~\eqref{Limiting2} are explicitly known and they all have the same mass: $\int_{\mathbb R^2} h(p) e^U = 8 \pi$. Therefore the value of $\beta_p$ in \eqref{1011}  just represents the sum of the masses $8\pi$ carried by each of such sharp peaks collapsing in $p$.

\medskip \noindent When $n\geq 3$ a similar classification result for solutions $U$ of
\begin{equation} \label{Limitingn}
    - \Delta_n U = h(p) e^U \hbox{ in }\mathbb{R}^n,
    \quad
    \int_{\mathbb R^n} e^U < \infty,
\end{equation}
with $h(p)>0$, has been recently provided by the first author in \cite{Esp}. For later convenience, observe in particular that the unique solution to
\begin{equation} \label{limitpb}
   -\Delta_n U=h(p) e^U \quad \hbox{ in }\mathbb{R}^n, \quad U \leq U(0)=0, \quad \int_{\mathbb{R}^n} e^U<+\infty ,
\end{equation}
is given by 
\begin{equation} \label{eq:Bubble}
   U(y)=-n \log \left(1+c_n^{-\frac{1}{n-1}} h(p)^{\frac{1}{n-1}}|y|^{\frac{n}{n-1}}\right) 
\end{equation} 
and satisfies
\begin{equation} \label{quantization}
    \int_{\mathbb{R}^n} h(p) e^U =c_n \omega_n , \quad c_n =n(\frac{n^2}{n-1})^{n-1}.
\end{equation}
Due to the invariance of \eqref{Limitingn} under translations and scalings,  all solutions to \eqref{Limitingn} are  given by the $(n+1)-$parameter family
$$
   U_{a, \lambda}  (y) = U \left(\lambda (y - a) \right) + n \log \lambda =   \log   \frac{ \lambda^n}{ ( 1 +  c_n^{-\frac{1}{n-1}} h(p)^{\frac{1}{n-1}} \lambda^{\frac{n}{n-1}} |y - a |^{\frac{n}{n-1}} )^n }  ,\quad (a,\lambda) \in \mathbb R^n \times (0,\infty),
$$
and satisfy $\int_{\mathbb R^n} h(p) e^{U}=c_n \omega_n $. As a by-product, under the condition 
\eqref{eq:Peso1} we necessarily have in \eqref{253} that 
\begin{equation} \label{1718}
\beta_p \geq c_n \omega_n, \end{equation}
a bigger value than the one appearing in the alternative (iii) of the Concentration-Compactness Principle.

\medskip
The blow-up mechanism that leads to the quantization result \eqref{1011} relies on an almost scaling-invariance property of the corresponding PDE, which guarantees that all the involved sharp peaks carry the same mass. Since it is also shared by the $n$-Liouville equation, a similar quantization property is expected to hold for the quasilinear case too:
\begin{equation} \label{10111}
\beta_p \in c_n\omega_n \mathbb{N}.
\end{equation}
However, the main point in proving \eqref{1011} is the limiting vanishing of the mass contribution coming from the neck regions between the sharp peaks. In the two-dimensional situation such crucial property follows by a Harnack inequality of $\sup+\inf$ type, established first in \cite{Shafrir} through an isoperimetric argument and an analysis of the mean average for a solution $u$ to \eqref{Liouville-Quasilinear}$_{n=2}$. A different proof can be given according to \cite{Rob} through Green's representation formula, see Remark \ref{commento} for more details, and a sharp form of such inequality has been later established in \cite{BrLiSh,ChenLin,ChenLi1} via isoperimetric arguments or moving planes/spheres techniques. However, all such approaches are not operating for the $n-$Liouville equation due to the nonlinearity of the differential operator; for instance, the Green representation formula is not anymore available in the quasilinear context and in the nonlinear potential theory an alternative has been found \cite{KiMa} in terms of the Wolff potential, which however fails to provide sharp constants as needed to derive precise asymptotic estimates on blowing-up solutions to \eqref{758}-\eqref{758bis}. We refer the interested reader to \cite{HKM,KuMi} for an overview on the nonlinear potential theory.
 
\medskip In order to establish the validity of \eqref{10111}, the first main contribution of our paper is represented by a new and very simple blow-up approach to $\sup+\inf$ inequalities. Since the limiting profiles have the form \eqref{eq:Bubble}, near a blow-up point we are able to compare in an  effective way a blowing-up sequence $u_k$ with the radial situation, in which sharp constants are readily available. Using the notations \eqref{1059} and \eqref{1100} our first main result reads as follows:
\begin{thm} \label{711intro} 
Given $0<a\leq b <\infty$, let $  \Lambda \subset \Lambda_{a,b}$ be a set which is equicontinous at each point of $\Omega$ and consider
$$   {\mathcal U} := 
   \{ u \in C^{1,\alpha} (\Omega) \, : \, u \hbox{ solves } \eqref {Liouville-Quasilinear} \hbox{ with }  h \in \Lambda \}.$$
Given a compact set $K \subset \Omega$ and $C_1>n-1$, then there exists $C_2 =C_2( \Lambda, K, C_1) >0$ so that
\begin{equation}\label{317bis}
        \max_K u+C_1 \inf_\Omega u \leq C_2,
        \quad
        \forall u \in {\mathcal U}.
\end{equation}
In particular, the inequality \eqref{317bis} holds for the solutions $u$ of \eqref{Liouville-Quasilinear} with $h \in \Lambda_{a,b}'$.
\end{thm}
By combining the $\sup+ \inf-$inequality  with a careful blow-up analysis, we are able to prove our second main result:
\begin{thm} \label{thm:Main}
Let $u_k$ be a sequence of solutions to \eqref{758} so that \eqref{758bis}-\eqref{253} hold. If one assumes~\eqref{eq:Peso1}, then $h(p)>0$ and $\beta_p $ satisfies \eqref{10111} for any $p \in \mathcal S$.
\end{thm} 

Our paper is structured as follows. Section~\ref{sec:SupInf} is devoted to establish the $\sup+\inf$ inequality.  Starting from a basic description of the blow-up mechanism, reported in the appendix for reader's convenience, a refined asymptotic analysis is carried over in Section~\ref{sec:Simple} to establish Theorem~\ref{thm:Main} when the blow-up point is ``isolated", according to some well established terminology as for instance in \cite{YYLi}. The quantization result in its full generality will be the object of Section~\ref{sec:Quantization}.

%%%%%%%%%%%%%%%%%%%%%%%%%%%%
\section{The $\sup+\inf$ inequality} \label{sec:SupInf}
\noindent 
%%%%%%%%%%%%%%%%%%%%%%%%%%%%

When $n=2$ the so-called ``sup+inf" inequality has been first derived  by Shafrir~\cite{Shafrir}: given $a,b>0$ and $K \subset \Omega$ a non-empty compact set, there exist constants $C_1,C_2> 0 $ so that 
\begin{equation} \label{0317}
\sup_K u + C_1 \inf_{\Omega} u \leq C_2
\end{equation}
does hold for any solution $u$ of \eqref{Liouville-Quasilinear}$_{n=2}$ with $0<a\leq h \leq b$ in $\Omega$; moreover one can take $C_1=1$ if $h\equiv 1$. Later on, Brezis, Li and Shafrir showed \cite{BrLiSh} the validity of \eqref{0317} in its sharp form with $C_1=1$ for any $h \in \Lambda_{a,b}'$, $a>0$. 

\medskip \noindent To reach this goal, a first tool needed is a general Harnack inequality that holds for solutions $u$ of $-\Delta_n u=f \geq 0 $ in $\Omega$. By means of the so-called nonlinear Wolff potential in \cite{KiMa} it is proved that there exists a constant $c_1>0$ such that
$$u(x)-\inf_{\Omega} u \geq c_1 \int_0^{\delta} \Big[ \int_{B_t(x)} f \Big]^{\frac{1}{n-1}} \frac{dt}{t} $$
holds for each ball $B_{2\delta}(x) \subset \Omega$. Since $f \geq 0$, note that the above inequality implies that
\begin{equation} \label{0617bis}
u(x)-\inf_{\Omega} u \geq c_1 \Big[ \int_{B_r(x)} f \Big]^{\frac{1}{n-1}} \log \frac{\delta}{r}
\end{equation}
for all $0<r<\delta$. The constant $c_1$ is not explicit and our argument could be significantly simplified if we knew $c_1= (n \omega_n)^{-\frac{1}{n-1}}$ , see Remark \ref{commento} for a thourough discussion. However,  in the class of radial functions, the following  lemma shows that indeed \eqref{0617bis} holds with the sharp constant $c_1=(n \omega_n)^{-\frac{1}{n-1}}$:
\begin{lm} \label{radial}
Let $u \in C^1 (B_{R_2}(a))$ and $0\leq f \in C(\overline{B_{R_2}(a)}) $ a radial function with respect to $a \in \mathbb{R}^n$ so that
$$-\Delta_n u \geq f \quad \hbox{in } B_{R_2}(a).$$
Then
\begin{equation} \label{1122}
u(a)-\inf_{B_{R_2}(a)} u \geq  (n \omega_n)^{-\frac{1}{n-1}} \int_0^{R_2} \left( \int_{B_{t} (a)} f   \right)^{\frac{1}{n-1}} \frac{dt}{t}.
\end{equation}
In particular, for each $0<R_1<R_2$ there holds
$$u(a)-\inf_{B_{R_2}(a)} u \geq (n \omega_n)^{-\frac{1}{n-1}} \left(  \int_{B_{R_1} (a)} f   \right)^{\frac{1}{n-1}} \log \frac{R_2}{R_1}.$$
\end{lm}
\begin{proof} Consider the radial solution $u_0$ solving
$$-\Delta_n u_0=f \hbox{ in }B_{R_2}(a),\quad u_0=0 \hbox{ on }\partial B_{R_2}(a).$$
Since 
$$-\Delta_n u \geq -\Delta_n u_0 \hbox{ in }B_{R_2}(a),\quad u-\inf_{B_{R_2}(a)} u \geq u_0 \hbox{ on }\partial B_{R_2}(a),$$
by comparison principle there holds
\begin{equation} \label{1127}
u-\inf_{B_{R_2}(a)} \geq u_0 \qquad \hbox{in }B_{R_2}(a).
\end{equation}
Furthermore, $u_0$ is radial with respect to $a$ and can be explicitly written as $(r=|x-a|)$:
\begin{equation} \label{1130}
u_0(r)=\int_r^{R_2} \left( \int_0^t s^{n-1}f(s) ds  \right)^{\frac{1}{n-1}} \frac{dt}{t} =
\int_r^{R_2} \left( \frac{1}{n \omega_n}  \int_{B_t(a)} f  \right)^{\frac{1}{n-1}} \frac{dt}{t} .
\end{equation}
By \eqref{1127}-\eqref{1130} we deduce the validity of \eqref{1122}. Since the function $t \to \int_{B_t(a)}f$ is non decreasing in view of $f \geq 0$, we have for each $0<R_1<R_2$ that
\begin{eqnarray*} 
u_0(a) \geq (n \omega_n)^{-\frac{1}{n-1}}
 \int_{R_1}^{R_2} \left( \int_{B_t(a)} f   \right)^{\frac{1}{n-1}} \frac{dt}{t} 
 \geq (n \omega_n)^{-\frac{1}{n-1}}
   \left( \int_{B_{R_1} (a)} f   \right)^{\frac{1}{n-1}} \log \frac{R_2}{R_1}
\end{eqnarray*}
and the proof is complete thanks to \eqref{1127}.
\end{proof}

\medskip This lemma is helpful to extend \eqref{0317} to the quasilinear case for all $C_1>n-1$ and this will be enough to establish the quantization result \eqref{10111}. It is an interesting open question to know whether or not the sharp inequality with $C_1=n-1$ is valid for a reasonable class of weights $h$, as when $n=2$ \cite{BrLiSh}. 
The $\sup+\inf$ inequality in Theorem \ref{711intro} will be an immediate consequence of the following result.
\begin{thm} \label{711} 
Let $0 < a \leq b < \infty$, consider the sets  $\Lambda$ and ${\mathcal U}$ defined in Theorem \ref{711intro}. Then given 
$K \subset \Omega$ a nonempty compact set and $C_1>n-1$, there exists a constant $C_3 >0$ such that 
$\displaystyle \max_K u \leq C_3$ holds for all $u \in {\mathcal U}$ satisfying 
$\displaystyle \max_K u + C_1 \inf_\Omega u \geq 0$
(Theorem \ref{711intro} follows by taking $C_2 = C_1 + C_3$). 
\end{thm}
\begin{proof} Choose $\delta>0$ so that $K_\delta=\{ x \in \Omega:\ \hbox{dist}(x,K) \leq 2 \delta\} \subset \Omega$. 
Let $u$ be a solution to \eqref{Liouville-Quasilinear} with $h \geq 0$ so that
\begin{equation} \label{1253}
    \max_K u+C_1 \inf_\Omega u \geq 0.
\end{equation}
Denote by $\bar x \in K$ a maximum point of $u$ in $K$: $u(\bar x)=\displaystyle \max_K u$. 
It follows from \eqref{0617bis} that
\begin{equation} \label{0617}
u(\bar x)-\inf_{\Omega} u \geq  c_1 \Big[ \int_{B_r(\bar x)} h e^{u} \Big]^{\frac{1}{n-1}} \log \frac{\delta}{r} 
\end{equation}
for all $0<r<\delta$ in view of $B_{2\delta} (\bar x) \subset \Omega$. Therefore, we deduce that
    $$ c_1  \Big[  \int_{B_r(\bar x)} h e^{u} \Big] ^{\frac{1}{n-1}}  
     \leq 
      \left\{ 
     \frac{u(\bar x)-\displaystyle \inf_{\Omega} u}{\log \frac{\delta}{r}} 
     \right\} \leq (1+\frac{1}{C_1}) \frac{u(\bar x)}{\log \frac{\delta}{r}}
$$
for all $0<r<\delta$ in view of \eqref{1253}-\eqref{0617}.

\medskip \noindent Arguing by contradiction, if the conclusion of the theorem is wrong, we can find a sequence 
$u_k \in {\mathcal U}$ satisfying \eqref{1253} such that, as $k \to \infty$, we have
\begin{equation} \label{blow}
      \max_K u_k \to +\infty.
\end{equation}
Letting $\bar x_k \in K$ so that $u_k(\bar x_k)=\displaystyle \max_K u_k$ and $\bar \mu_k =e^{-\frac{u_k(\bar x_k)}{n}}$,  we have that $\bar \mu_k \to 0$ as $k \to +\infty$ in view of \eqref{blow}. Since for each $R>0$ we can find $k_0 \in \mathbb{N}$ so that $R\bar \mu_k < \delta$ for all $k \geq k_0$,  by \eqref{0617} we deduce that
 \begin{equation}  \label{1224}      
  c_1 \limsup_{k \to \infty} 
\Big[  \int_{B_{R \bar \mu_k}(\bar x_k)} h_k e^{u_k} \Big] ^{\frac{1}{n-1}}  
   \leq n \left( 1 + \frac{1}{C_1}  \right).
 \end{equation}
By applying Ascoli-Arzela, we can further assume, up to a subsequence, that 
\begin{equation} \label{0437}
\bar x_k \to p \in K, \quad h_k \to h  \geq a>0 \hbox{ in }C_{loc}(\Omega)
\end{equation}
as $k \to +\infty$.

\medskip \noindent Once \eqref{1224} is established, in order to reach a contradiction we aim to replace $\bar x_k$ by a nearby local maximum point $x_k \in \Omega$ of $u_k$ with $u_k(x_k)\geq u_k(\bar x_k)=\displaystyle \max_K u_k$. We can argue as follows: the function $\bar U_k (y) = u_k(\bar \mu_k y+\bar x_k)+n\log \bar \mu_k$ satisfies  
$$
   - \Delta_n \bar U_k = h_k (\bar \mu_k y+\bar x_k) e^{\bar U_k}
   \hbox{ in }  \Omega_k=\frac{\Omega-\bar x_k}{ \bar \mu_k}
$$
and
 \begin{equation} \label{1058bis}
  \bar U_k  \leq \bar U_k (0) = 0 \, \,  \hbox{ in } \frac{K - \bar x_k}{\bar \mu_k} ,
   \qquad
   \limsup_{k \to +\infty} \int_{B_R(0)} e^{\bar U_k}\leq  \frac{1}{a}   \big( \frac{n}{c_1} \big)^{n-1}  \left( 1 + \frac{1}{C_1}  \right)^{n-1}   
\end{equation}
in view of \eqref{1224}-\eqref{0437}. From \eqref{1058bis}, the Concentration-Compactness Principle and $\bar U_k(0)=0$ we deduce, up to a subsequence, that:
 
 \medskip
 
\begin{enumerate}
\item[{\bf (i)}]
either, $\bar U_k$ is bounded in $L^\infty_{loc}(\mathbb{R}^n)$
\item[{\bf (ii)}]
or, $h_k (\bar \mu_k y+\bar x_k)  e^{\bar U_k} \rightharpoonup \beta_0 \delta_0+\displaystyle \sum_{i=1}^I \beta_i \delta_{p_i}$ weakly in the sense of measures in $\mathbb{R}^n$, for some $\beta_i \geq n^n \omega_n$, $i \in \{0,\ldots, I\}$, and distinct points $p_1,\ldots,p_I \in \mathbb{R}^n \setminus \{0\}$, and $\bar U_k \to -\infty$ locally uniformly in $\mathbb{R}^n \setminus\{0, p_1,\ldots,p_I\}$.
\end{enumerate}

%%%%%%%%%%%%%%%%%%%%%%%%%%%%%%%%%
\medskip \noindent \underline{\bf{Case (i)}}: $\bar U_k$ is bounded in $L^\infty_{loc}(\mathbb{R}^n)$

\medskip \noindent 
By elliptic estimates \cite{Dib,Tol} we deduce that $ \bar U_k   \to \bar U $  in $C^1_{loc}(\mathbb{R}^n)$ as $k\to +\infty$, where $\bar U$ satisfies \eqref{Limitingn} with $h(p)>0$ and $\bar U(0)=0$ in view of \eqref{0437}-\eqref{1058bis}. Since in general $\frac{K-\bar x_k}{\bar \mu_k}$ does not tend to $\mathbb{R}^n$ as $k \to +\infty$, by \eqref{1058bis} we cannot guarantee that $\bar U$ achieves the maximum value at $0$. However, by the classification result in \cite{Esp} we have that $\bar U=U_{a,\lambda}$ for some $(a,\lambda)\in \mathbb{R}^n \times (0,\infty)$. Since $\bar U$ is a radially strictly decreasing function with respect to $a$, we can find a sequence $a_k \to a$ such that as $k \to +\infty$
\begin{equation}\label{804}
      \bar U_k(a_k)=\max_{B_R(a_k)}\bar U_k,
      \quad
      \bar U_k (a_k) \to \bar U (a) =\max_{\mathbb{R}^n}  \bar U
\end{equation}
for all $R>0$ and $k$ large (depending on $R$). Setting $x_k=\bar \mu_k a_k+\bar x_k$ and $\mu_k=e^{-\frac{u_k(x_k)}{n}}$, we have that 
$$
   u_k(x_k)=\bar U_k(a_k)-n\log \bar \mu_k \geq \bar U_k(0)-n \log \bar \mu_k=u_k(\bar x_k)$$ 
 and
\begin{equation} \label{750}
1  \leq \frac{\bar \mu_k}{\mu_k}=e^{\frac{u_k(x_k)-u_k(\bar x_k)}{n}}=e^{\frac{\bar U_k(a_k)}{n}} \xrightarrow{ \, k \to \infty \,}
e^{\frac{\max_{\mathbb{R}^n}  \bar U}{n}} 
\end{equation}
in view of \eqref{804}. Let us now rescale $u_k$ with respect to $x_k$ by setting 
$$U_k(y)=u_k(\mu_k y+x_k)+n\log \mu_k.
$$
Since  \eqref{1058bis}-\eqref{804} re-write in terms of $U_k$ as
\begin{eqnarray}\label{1020}
 && \limsup_{k \to +\infty} \int_{B_{R \frac{\bar \mu_k}{\mu_k}}(- \frac{\bar \mu_k}{\mu_k} a_k)} e^{U_k}
\leq  \frac{1}{a}   \big( \frac{n}{c_1} \big)^{n-1}  \left( 1 + \frac{1}{C_1}  \right)^{n-1} \\
\label{2146}
&& U_k(0)=\max_{B_{R \frac{\bar \mu_k}{\mu_k}}(0)}U_k=0
\end{eqnarray}
for all $R>0$, thanks to the uniform convergence \eqref{0437}, by \eqref{750}-\eqref{2146} and elliptic estimates \cite{Dib,Tol} we have that $U_k   \to U $ in $C^1_{loc}(\mathbb{R}^n)$ as $k\to +\infty$, where $U$ satisfies \eqref{limitpb} with $h(p)>0$. Then $U$ takes precisely the form \eqref{eq:Bubble} and satisfies \eqref{quantization}. 

\medskip \noindent Therefore for each $R >0$ and $\epsilon\in (0,1)$, there exists $k_0 = k_0 (R, \varepsilon) >0$ so that for all $k \geq k_0$ there hold $B_{R\mu_k}(x_k) \subset B_{\delta}(x_k) \subset B_{2\delta}(\bar x_k)$ and
\begin{equation} \label{1159bis}
     h_k(x) \geq  \sqrt{1-\epsilon} \ h(p),  \:\:      u_k(x) \geq  U_{x_k, \mu_k^{-1} } +  \log \sqrt{1-\epsilon}  \qquad \hbox{in }B_{R \mu_k}(x_k)
\end{equation}
in view of \eqref{0437} and $U_k \geq U +  \log \sqrt{1-\epsilon}$ in $B_{R}(0)$. Setting $f_k (t) =  (1-\epsilon)    h(p) e^{U_{ x_k, \mu_{k}^{-1}  } }  \chi_{B_{R \mu_k}(x_k)}$, by \eqref{1159bis} we have that $h_k e^{u_k}\geq f_k$ in $B_{\delta}(x_k)$ and then Lemma \ref{radial} implies the following lower bound for all $k \geq k_0$: 
\begin{eqnarray} 
u_k(x_k)-\inf_{B_{\delta}(x_k)}u_k \geq   \left(  \frac{1-\epsilon}{n \omega_n} \int_{B_{R} (0)} h(p) e^U   \right)^{\frac{1}{n-1}} \log \frac{ \delta}{R \mu_k}
  \label{1231}
\end{eqnarray}
in view of $\int_{B_{R\mu_k} (x_k)} f_k= (1-\epsilon)    \int_{B_R(0)} h(p) e^U$. Recalling that $\mu_k=e^{-\frac{u_k(x_k)}{n}}$, by \eqref{1231} we deduce that
$$
    \left(   \frac{1-\epsilon}{n \omega_n} \int_{B_{R} (0)} h(p) e^U   \right)^{\frac{1}{n-1}}  
    \leq 
    \frac{u_k(x_k)-\inf_{B_{\delta}(x_k)} u_k}{ \log \frac{\delta}{R} +  \frac{ u_k (x_k)}{n}  }.
$$
Since
$$u_k(x_k)+C_1 \inf_{B_{\delta}(x_k)} u_k \geq u_k(\bar x_k)+C_1 \inf_\Omega u_k=\max_K u_k+C_1 \inf_\Omega u_k\geq 0$$
in view of  \eqref{1253}, letting $k \to \infty$ we deduce
$$
   \left(  \frac{1-\epsilon}{n \omega_n} \int_{B_{R} (0)} h(p) e^U   \right)^{\frac{1}{n-1}}  
    \leq 
     n \limsup_{k \to \infty} \left\{ 1 -\frac{\inf_{B_{\delta}(x_k)} u_k}{u_k (x_k)} \right\}
     \leq 
     n \left( 1 + \frac{1}{C_1} \right) .    
$$
 Since this holds for each $R, \varepsilon >0$ we deduce that 
$$
      \frac{1}{n \omega_n} \int_{\mathbb{R}^n }  h(p) e^U  \leq \left[n ( 1 + \frac{1}{C_1}) \right]^{n-1}
     <  \left( \frac{n^2}{n-1} \right)^{n-1}
$$
in view of the assumption $C_1>n-1$. On the other hand, by \eqref{quantization} the left hand side is precisely $\left( \frac{n^2}{n-1} \right)^{n-1}$ and this is a contradiction.

\medskip \noindent 
\underline{\bf{Case (ii)}}: $h_k (\bar \mu_k y+\bar x_k)  e^{\bar U_k} \rightharpoonup \beta_0 \delta_0+\displaystyle \sum_{i=1}^I \beta_i \delta_{p_i}$ weakly in the sense of measures in $\mathbb{R}^n$, for some $\beta_i \geq n^n \omega_n$, $i \in \{0,\ldots, I\}$, and distinct points $p_1,\ldots,p_I \in \mathbb{R}^n \setminus \{0\}$, and $\bar U_k \to -\infty$ locally uniformly in $\mathbb{R}^n \setminus\{ 0, p_1,\ldots,p_I\}$

\medskip \noindent If $I \geq 1$, w.l.o.g. assume that $p_1,\dots,p_I \notin \overline{B_1(0)}$. Since $\bar U_k \to -\infty$ locally uniformly in $\overline{B_1(0)} \setminus \{0\}$ and $\displaystyle \max_{B_1(0)} \bar U_k \to +\infty$ as $k \to +\infty$, we can find $a_k \to 0$ so that
\begin{equation} \label{903}
\bar U_k(a_k)=\displaystyle \max_{B_1(0)} \bar U_k \to +\infty
\end{equation} 
as $k \to +\infty$. We now argue in a similar way as in case (i). Setting $x_k=\bar \mu_k a_k+\bar x_k$ and $\mu_k=e^{-\frac{u_k(x_k)}{n}}$, we have that $u_k(x_k)=\bar U_k(a_k)-n\log \bar \mu_k \geq \bar U_k(0)-n \log \bar \mu_k=u_k(\bar x_k)$ and
\begin{equation} \label{905}
\frac{\bar \mu_k}{\mu_k}=e^{\frac{\bar U_k(a_k)}{n}} \to +\infty
\end{equation}
as $k \to +\infty$ in view of \eqref{903}. Setting 
$$U_k(y)=u_k(\mu_k y+x_k)+n\log \mu_k,
$$
by \eqref{1058bis} and \eqref{903} we have that
\begin{eqnarray} \label{1020bis}
 && \limsup_{k \to +\infty} \int_{B_{R \frac{\bar \mu_k}{\mu_k}}(- \frac{\bar \mu_k}{\mu_k} a_k)} e^{U_k}
\leq  \frac{1}{a}   \big( \frac{n}{c_1} \big)^{n-1}  \left( 1 + \frac{1}{C_1}  \right)^{n-1} \\
\label{2146bis}
&& U_k(0)=\max_{B_{\frac{\bar \mu_k}{\mu_k}}(0)}U_k=0
\end{eqnarray}
for all $R>0$. Since $B_R(0) \subset B_{R \frac{\bar \mu_k}{\mu_k}}(- \frac{\bar \mu_k}{\mu_k} a_k) $ for all $k$ large in view of \eqref{905} and $\displaystyle \lim_{k\to +\infty}a_k=0$, by \eqref{905}-\eqref{2146bis} and elliptic estimates \cite{Dib,Tol} we have  that $U_k   \to U $  in $C^1_{loc}(\mathbb{R}^n)$ as $k\to +\infty$, where $U$ satisfies \eqref{limitpb}-\eqref{quantization}. We now proceed exactly as in case (i) to reach a contradiction. The proof is complete.
\end{proof}
\begin{oss} \label{commento} When $n=2$, the ``sup+inf" inequality was first derived by Shafrir~\cite{Shafrir} through an isoperimetric argument. It becomes clear in \cite{Rob}, when dealing with a fourth-order exponential PDE in $\mathbb{R}^4$, that the main point comes from the linear theory, which allows there to avoid the extra work needed in our framework. For instance, in the two dimensional case, 
inequality \eqref{0617bis}  is an easy consequence of the Green representation formula: 
given a solution $u$ to $-\Delta u=f$ in a domain containing $B_1(0)$, we can use the fundamental solution of the Laplacian to obtain
$$u(x)-\inf_{B_1(0)} u \geq -\frac{1}{2\pi}\int_{B_1(0)} \log \frac{|x-y|}{||x|y-\frac{x}{|x|}|} f(y) \qquad \forall \ x \in B_1(0),$$
which  through an integration by parts gives
$$ u(0)-\inf_{\Omega} u \geq -\frac{1}{2\pi}\int_{B_1(0)} \log |y| f(y) = \frac{1}{2\pi} \int_0^1 [\int_{B_t(0)} f] \frac{dt}{t}.$$
This linear argument also provides the optimal constant $c_1=\frac{1}{2\pi}$, which can be exploited to simplify the proof of Theorem \ref{711} as follows. The estimates~\eqref{1224} re-writes as
\begin{eqnarray} \label{1907}
\limsup_{k \to +\infty}  \int_{B_R(0)} h_k(\bar \mu_ky+\bar x_k) e^{\bar U_k} &=&\limsup_{k \to +\infty}  \int_{B_{R\bar \mu_k}(\bar x_k)} h_k e^{u_k} 
  \leq \left[\frac{n}{c_1}(1+\frac{1}{C_1})\right]^{n-1}\\
  & < & \Big[  \frac{n^2}{ c_1 (n-1)} \Big]^{n-1} \nonumber 
\end{eqnarray}
for all $R>0$ when $C_1 > n-1$. Since $c_1 = \frac{1}{2 \pi}$ and $[\frac{n^2}{ c_1 (n-1)}]^{n-1}=8\pi$ when $n =2$,  in case (i) of the above proof we deduce that $\displaystyle \int_{\mathbb{R}^2} h(p) e^{\bar U} <8\pi$, in contrast with the quantization property $\int_{\mathbb{R}^2} h(p) e^{U}=8\pi$ for every solution $U$ of \eqref{Limiting2}. Assuming w.l.o.g. $p_1,\dots,p_I \notin \overline{B_1(0)}$ if $I\geq 1$, in case (ii) of the above proof we deduce from \eqref{1907} with $R=1$ that $\beta_0 <8\pi$, in contrast with the lower estimate $\beta_0\geq 8\pi$ coming from  \eqref{eq:Peso1} and \eqref{1718} when $n=2$. Therefore, the proof of Theorem \ref{711} in dimension two becomes considerably simpler. 

\medskip
When $n \geq 3$ Green's representation formula is not available for $\Delta_n$ and \eqref{0617bis} does hold \cite{KiMa} with some constant $0<c_1 \leq (n \omega_n)^{-\frac{1}{n-1}}$. Since $c_1$ is in general strictly below the optimal one $(n \omega_n)^{-\frac{1}{n-1}}$, we need to fill the gap thanks to the exponential form of the nonlinearity through a blow-up approach. With this strategy a comparison argument with the radial case is exploited, since in the radial context inequality \eqref{0617bis} does hold with optimal constant $c_1=(n \omega_n)^{-\frac{1}{n-1}}$ thanks to Lemma \ref{radial}.
\end{oss}

As a consequence of the $\sup+\inf$ estimates in Theorem \ref{711}, we deduce the following useful decay estimate:
\begin{cor} \label{decay} Let $u_k$ be a sequence of solutions to \eqref{758}, satisfying \eqref{eq:Peso1} with $h_k \geq \epsilon_0 >0$ in  $B_{4r_0}(x_k) \subset \Omega$ and
\begin{equation} \label{1311}
|x-x_k|^n e^{u_k}  \leq C \qquad \hbox{in }B_{2 b_k}(x_k) \setminus B_{a_k}(x_k)
\end{equation}
for $0<2a_k< b_k \leq 2 r_0$. Then, there exist $\alpha,C>0$ such that
\begin{equation}\label{1315}
u_k \leq C-\frac{\alpha}{n} u_k(x_k)-(n+\alpha)\log |x-x_k|  
\end{equation}
for all $2a_k \leq |x-x_k| \leq b_k$. In particular, if $e^{-\frac{u_k(x_k)}{n}}=o(a_k)$ as $k \to +\infty$ we have that
\begin{equation} \label{1703}
\lim_{k \to +\infty} \int_{B_{b_k}(x_k) \setminus B_{2a_k}(x_k)} h_ke^{u_k} =0.
\end{equation}
\end{cor}
\begin{proof} Letting $V_k(y)=u_k(ry+x_k)+n\log r$ for any $0<r\leq b_k$, we have that  $-\Delta_n V_k=h_k(ry+x_k) e^{V_k}$ does hold in $\Omega_k=\frac{\Omega-x_k}{r}$ and \eqref{1311} implies that 
\begin{equation} \label{1311 1}
 \sup_{B_2(0) \setminus B_{\frac{1}{2}}(0)} |y|^n e^{V_k}\leq C <+\infty
\end{equation}
for all $2a_k \leq r\leq b_k$. Since $V_k$ is uniformly bounded from above in $B_2(0) \setminus B_{\frac{1}{2}}(0)$ in view of \eqref{1311 1}, by the Harnack inequality \cite{Ser1,Tru} it follows that there exist $C>0$ and $C_0 \in (0,1]$ so that
\begin{equation} \label{1539}
C_0 \sup_{|y|=1}V_k\leq \inf_{|y|=1} V_k+C
\end{equation}
for all $2a_k\leq r \leq b_k$. 

\medskip \noindent Up to a subsequence, assume that $\displaystyle \lim_{k \to +\infty} x_k =x_0$. By assumption we have that $h_k(ry+x_k) \to h(ry+x_0)\geq \epsilon_0>0$ in $C_{loc}(B_1(0))$ as $k \to +\infty$  for all $0<r\le 2 r_0$. For any given $C_1>n-1$,  by Theorem \ref{711intro} applied to $V_k$ in $B_1(0)$ with $K=\{0\}$ we obtain the existence of $C_2>0$ so that
\begin{equation} \label{1546}
V_k(0)+C_1 \inf_{B_1(0)} V_k= V_k(0)+C_1 \inf_{|y|=1} V_k \leq C_2
\end{equation}
does hold for all $k$ and all $0<r\leq 2 r_0$. Inserting \eqref{1546} into \eqref{1539} we deduce that
$$\sup_{|y|=1}V_k  \leq C-\frac{\alpha}{n} V_k(0)$$
for all $2a_k\leq r \leq b_k$, with $\alpha=\frac{n}{C_0 C_1}>0$ and some $C>0$, which re-writes in terms of $u_k$ as \eqref{1315}. In particular, by \eqref{1315} we deduce that
$$0\leq \int_{2a_k\leq |x-x_k|\leq b_k} h_ke^{u_k} \leq C e^{-\frac{\alpha}{n}u_k(x_k)}\int_{2a_k\leq |x-x_k|\leq b_k} \frac{dx}{|x-x_k|^{n+\alpha}}=\frac{C n \omega_n}{\alpha 2^{\alpha}} [a_k e^{\frac{u_k(x_k)}{n}} ]^{-\alpha}\to 0$$
provided $e^{-\frac{u_k(x_k)}{n}}=o(a_k)$ as $k \to +\infty$, in view of \eqref{eq:Peso1} and $B_{4r_0}(x_k) \subset \Omega$.
\end{proof}

%%%%%%%%%%%%%%%%%%%%%%%%%%%%
\section{The case of isolated blow-up} \label{sec:Simple}
\noindent 
%%%%%%%%%%%%%%%%%%%%%%%%%%%%

\noindent The following basic description of the blow-up mechanism is very well known, see \cite{LiSh} in the two-dimensional case and for example \cite{DHR} in a related higher-dimensional context, and is the starting point for performing a more refined asymptotic analysis. For reader's convenience its proof is reported in the appendix.
\begin{thm} \label{352}
Let $u_k$ be a sequence of solutions to \eqref{758} which satisfies \eqref{758bis} and
\begin{equation}\label{824}
h_k e^{u_k} \rightharpoonup \beta \delta_0 \hbox{ weakly in the sense of measures in }B_{3\delta}(0) \subset \Omega \end{equation}
for some $\beta>0$ as  $k \to \infty$. Assuming \eqref{eq:Peso1}, then $h(0)>0$ and, up to a subsequence, we can find a finite number of points $x_k^1,\dots,x_k^N$ so that for all $i \not= j$
\begin{eqnarray}
&& |x_k^i| +\mu_k^i+ \frac{\mu_k^i+\mu_k^j}{|x_k^i-x_k^j|}\to 0  \label{339}\\
&&  u_k(\mu_k^i y+x_k^i)+n \log \mu_k^i \to U(y) \hbox{ in }C^1_{loc}(\mathbb{R}^n) \label{340}
\end{eqnarray}
as $k \to +\infty$ and
\begin{eqnarray}
\min\{ |x-x_k^1|^n,\dots,|x-x_k^N|^n \}  e^{u_k} \leq C \hbox{ in }B_{2 \delta}(0) \label{341}
\end{eqnarray}
for all $k$ and some $C>0$, where $U$ is given by \eqref{eq:Bubble} with $p=0$ and
\begin{equation} \label{828}
u_k(x_k^i)=\max_{B_{\mu_k^i}(x_k^i)} u_k, \quad \mu_k^i=e^{-\frac{u_k(x_k^i)}{n}}.
\end{equation}
\end{thm}
In this section we consider the case of an ``isolated" blow-up point corresponding to have $N=1$ in Theorem \ref{352}, namely
\begin{equation} \label{simple}
|x-x_k|^n e^{u_k} \leq C \hbox{ in }B_{2 \delta }(0) 
\end{equation}
for all $k$ and some $C>0$, where $x_k$ simply denotes $x_k^1$. The following result, corresponding to Theorem \ref{thm:Main} for the case of an isolated blow-up, extends the analogous two-dimensional one \cite[Prop. $2$]{LiSh} to $n \geq 2$. 
\begin{thm} \label{700}
Let $u_k$ be a sequence of solutions to \eqref{758}  which satisfies \eqref{758bis}, \eqref{eq:Peso1}, \eqref{824} and \eqref{simple}. Then 
$$\beta=c_n \omega_n.$$
\end{thm}
\begin{proof} First, notice that $x_k \to 0$ as $k \to +\infty$ and $h(0)>0$ in view of Theorem \ref{352}. Since $h \in C(\Omega)$ take $0<r_0 \leq \frac{\delta}{2}$ and $\epsilon_0>0$ so that $h \geq 2 \epsilon_0$ for all $y \in B_{5r_0}(0)$. By \eqref{eq:Peso1} we then deduce that $h_k \geq \epsilon_0>0$ in $B_{4r_0}(x_k)\subset \Omega$. Letting $\mu_k=e^{-\frac{u_k(x_k)}{n}}$ and $U_k=u_k(\mu_ky+x_k)+n \log \mu_k$, there holds
$$ \lim_{k \to +\infty} \int_{B_{R\mu_k}(x_k)} h_k e^{u_k} dx=\lim_{k \to +\infty}  \int_{B_R(0)}h_k(\mu_ky+x_k) e^{U_k} dy = \int_{B_R(0)} h(0)e^U dy$$
in view of \eqref{eq:Peso1}  and \eqref{340}. Therefore we can construct $R_k \to +\infty$ so that $R_k \mu_k \leq r_0$ and
\begin{equation} \label{1749}
\lim_{k \to +\infty} \int_{B_{R_k \mu_k}(x_k)} h_k e^{u_k}dx=c_n \omega_n
\end{equation}
in view of \eqref{quantization} with $p=0$. Since \eqref{simple} implies the validity of \eqref{1311} with $b_k=r_0$ and $a_k=\frac{R_k \mu_k}{2}$, we can apply Corollary \ref{decay} to deduce by \eqref{1703} that
\begin{equation} \label{1752}
\lim_{k \to +\infty} \int_{B_{r_0}(x_k) \setminus B_{R_k \mu_k}(x_k)} h_ke^{u_k} =0
\end{equation}
in view of $\mu_k=e^{-\frac{u_k(x_k)}{n}}=o(a_k)$ as $k \to +\infty$. Since by the Concentration-Compactness Principle we have that $u_k \to -\infty$ locally uniformly in $B_{3\delta}(0) \setminus \{0\}$ as $k \to +\infty$, we finally deduce that $\beta$ in \eqref{824} satisfies
$$\beta=\lim_{k \to +\infty} \int_{B_{r_0}(x_k)} h_k e^{u_k}=c_n \omega_n$$
in view of \eqref{1749}-\eqref{1752}, and the proof is complete.
\end{proof}

%%%%%%%%%%%%%%%%%%%%%%%%%%%%%%%%%%%%
\section{General quantization result} \label{sec:Quantization}
%%%%%%%%%%%%%%%%%%%%%%%%%%%%%%%%%%%%

\noindent In order to address quantization issues in the general case where $N\geq 2$ in Theorem \ref{352}, in the following result let us consider a more general situation.
\begin{thm} \label{3522}
Let $u_k$ be a sequence of solutions to \eqref{758} which satisfies \eqref{758bis} and \eqref{824}. Assume \eqref{eq:Peso1} and the existence of a finite number of points $x_k^1,\dots,x_k^N$  and radii $r_k^1,\dots,r_k^N$ so that  
for all $i \not= j$
\begin{eqnarray} \label{858}
|x_k^i| +\frac{\mu_k^i}{r_k^i} +\frac{r_k^i+r_k^j}{|x_k^i-x_k^j|}\to 0
\end{eqnarray} 
as $k\to +\infty$,  where $\mu_k^i=e^{-\frac{u_k(x_k^i)}{n}}$, and
\begin{eqnarray} 
\min\{ |x-x_k^1|^n,\dots,|x-x_k^N|^n \}  e^{u_k}  \leq C \hbox{ in }B_{2\delta}(0) \setminus  \bigcup_{i=1}^N B_{r_k^i}(x_k^i) \label{811}
\end{eqnarray}
for all $k$ and some $C>0$. If $\displaystyle \lim_{k \to +\infty} \int_{B_{2 r_k^i}(x_k^i)} h_k e^{u_k}=\beta_i$ for all $i=1,\dots,N$, then
\begin{equation} \label{tquant}
\lim_{k \to +\infty} \int_{B_{\frac{\delta}{2}}(0)} h_k e^{u_k}=\sum_{i=1}^N \beta_i.
\end{equation}
\end{thm}
\begin{proof}
First of all, by applying the Concentration-Compactness Principle to $u_k(r_k^i y+x_k^i)+n \log r_k^i$ we obtain that $\beta_i>0$, $i=1,\dots,N$, in view of $\frac{\mu_k^i}{r_k^i}\to 0$ as $k \to +\infty$. Since $h(0)>0$ by Theorem \ref{352} and $h \in C(\Omega)$, we can find $0<r_0 \leq \frac{\delta}{2}$ so that $h_k \geq \epsilon_0>0$ in $B_{4r_0}(x_k)\subset \Omega$ in view of \eqref{eq:Peso1}. The case $N=1$ follows the same lines as in Theorem \ref{700}: since \eqref{858}-\eqref{811} imply the validity of \eqref{1311} with $b_k=r_0$ and $a_k=r_k$,  by Corollary \ref{decay} we get that
$$\lim_{k \to +\infty} \int_{B_{r_0}(x_k) \setminus B_{2r_k}(x_k)} h_ke^{u_k} =0$$
in view of $\mu_k=o(r_k)$. Since $u_k \to -\infty$ locally uniformly in $B_{3\delta}(0) \setminus \{0\}$ as $k \to +\infty$ in view of the Concentration-Compactness Principle, \eqref{tquant} is then  established when $N=1$.

\medskip \noindent We proceed by strong induction in $N$ and assume the validity of Theorem \ref{3522} for a number of points $\leq N-1$. Given  $x_k^1,\dots,x_k^N$, define their minimal distance as $d_k=\min\{ |x_k^i-x_k^j|: \: i,j=1,\dots,N,\: i \not=j \}$. Since $B_{\frac{d_k}{2}}(x_k^i) \cap B_{\frac{d_k}{2}}(x_k^j)=\emptyset$ for $i \not= j$, we deduce that $|x-x_k^i|\leq |x-x_k^j|$ in $B_{\frac{d_k}{2}}(x_k^i)$ for all $i \not= j$ and then \eqref{811} gets rewritten as $ |x-x_k^i|^n e^{u_k}  \leq C$ in $B_{\frac{d_k}{2}}(x_k^i) \setminus B_{r_k^i}(x_k^i)$ for all $i=1,\dots,N$. By \eqref{858} and Corollary \ref{decay} with $b_k=\frac{d_k}{4}$ and $a_k=r_k^i$ we deduce that $\displaystyle \int_{B_{\frac{d_k}{4}}(x_k^i) \setminus B_{2 r_k^i}(x_k^i)} h_k e^{u_k}  \to 0$ as $k \to +\infty$ and then 
\begin{equation} \label{1038}
\lim_{k \to +\infty} \int_{B_{\frac{d_k}{4}}(x_k^i)} h_k e^{u_k}=\beta_i \qquad \forall \: i=1,\dots,N.
\end{equation}
Up to relabelling, assume that $d_k=|x_k^1-x_k^2|$ and consider the following set of indices 
$$I=\{i=1,\dots,N: \: |x_k^i-x_k^1| \leq C d_k \hbox{ for some }C>0\}$$
of cardinality $N_0 \in [2,N]$ since $1,2 \in I$ by construction. Up to a subsequence, we can assume that
\begin{equation} \label{1314}
\frac{|x_k^j-x_k^i|}{d_k} \to +\infty \hbox{ as }k\to +\infty
\end{equation}
for all $i \in I$ and $j \notin I$. Letting $\tilde u_k(y)=u_k(d_ky+x_k^1)+n \log d_k$, notice that 
\begin{equation} \label{1845}
\tilde u_k (\frac{x_k^i-x_k^1}{d_k})=u_k(x_k^i)+n \log d_k=n \log \frac{d_k}{\mu_k^i} \to +\infty
\end{equation}
as $k \to +\infty$ in view of \eqref{858}, and \eqref{811} re-writes as
\begin{equation} \label{1842}
\min\{ |y-\frac{x_k^i-x_k^1}{d_k}|^n: i \in I \}  e^{\tilde u_k}  \leq C_R \hbox{ uniformly in } B_R(0) \setminus  \bigcup_{i\in I} B_{\frac{r_k^i}{d_k}}(\frac{x_k^i-x_k^1}{d_k})
\end{equation}
for any $R>0$ thanks to \eqref{1314}. Since $\frac{r_k^i}{d_k} \to 0$ as $k \to +\infty$ in view of \eqref{858}, by \eqref{1845}-\eqref{1842} and the Concentration-Compactness Principle we deduce that 
$$\tilde u_k \to -\infty \hbox{ uniformly on }  B_R(0) \setminus  \bigcup_{i\in I} B_{\frac{1}{4}}(\frac{x_k^i-x_k^1}{d_k})$$
as $k \to +\infty$ and then
\begin{equation} \label{1041}
\lim_{k \to +\infty} \int_{B_{Rd_k}(x_k^1) \setminus \displaystyle \bigcup_{i \in I} B_{\frac{d_k}{4}}(x_k^i)} h_k e^{u_k}=0.
\end{equation}
By \eqref{1038} and \eqref{1041} we finally deduce that
$$ \lim_{k \to +\infty} \int_{B_{Rd_k}(x_k^1)} h_k e^{u_k}=\sum_{i \in I} \beta_i$$
since the balls $B_{\frac{d_k}{4}}(x_k^i)$, $i \in I$, are disjoint.

\medskip \noindent Set $x_k'=x_k^1$, $r_k'=\frac{Rd_k}{2}$ and $\beta'=\displaystyle \sum_{i \in I}\beta_i$. We apply the inductive assumption with the $N-N_0+1$ points $x_k'$ and $\{ x_k^j\}_{j \notin I}$, radii $r_k'$ and $\{r_k^j\}_{j \notin I}$, masses $\beta'$ and $\{\beta_j\}_{j \notin I}$ thanks to the following reduced form of assumption \eqref{811}:
$$\min\{ |x-x_k'|^n,\: |x-x_k^j|^n: \: j \notin I \}  e^{u_k}  \leq C \hbox{ in }B_{2\delta}(0) \setminus  [B_{r_k'}(x_k') \cup \bigcup_{j \notin I} B_{r_k^j}(x_k^j)]$$
provided $R$ is taken sufficiently large. It finally shows the validity of \eqref{tquant} for the index $N$, and the proof is achieved by induction.
\end{proof}

\medskip \noindent We are now in position to establish Theorem \ref{thm:Main} in full generality.
\begin{proof} We first apply Theorem \ref{352} to have a first blow-up description of $u_k$. We have that $\beta_p=N c_n \omega_n$ for all $p \in \mathcal S$ in view of Theorem \ref{3522}, provided we can construct radii $r_k^i$, $i=1,\dots,N$, satisfying \eqref{858} and
\begin{equation} \label{1615}
\lim_{k \to +\infty} \int_{B_{2 r_k^i}(x_k^i)} h_k e^{u_k}=c_n \omega_n.
\end{equation}
Since by \eqref{eq:Peso1} and \eqref{340} we deduce that
\begin{equation} \label{1619}
\int_{B_{R \mu_k^i}(x_k^i)} h_k e^{u_k} \to \int_{B_R(0)}h(p) e^U
\end{equation}
as $k \to +\infty$,  by \eqref{quantization} for all $i=1,\dots N$ we can find $R_k^i \to +\infty$ so that $R_k^i \mu_k^i \leq \delta$ and 
\begin{equation} \label{1931}
\int_{B_{2 R_k^i \mu_k^i}(x_k^i)} h_k e^{u_k} \to c_n \omega_n.
\end{equation}
If $N=1$ we simply set $r_k=R_k \mu_k$ (omitting the index $i=1$). When $N\geq 2$,  by \eqref{339} we deduce that $\mu_k^i=o(d_k^i)$, where $d_k^i=\min \{|x_k^j-x_k^i|: \: j  \not=i \}$, and we can set $r_k^i= \min\{ \sqrt{d_k^i \mu_k^i}, R_k^i \mu_k^i\}$ in this case. By construction the radii $r_k^i$ satisfy \eqref{858} and \eqref{1615} easily follows by \eqref{quantization} and \eqref{1619} and \eqref{1931}, in view of the chain of  inequalities 
$$\int_{B_{R \mu_k^i}(x_k^i)} h_k e^{u_k} \leq  \int_{B_{2 r_k^i}(x_k^i)} h_k e^{u_k} \leq \int_{B_{2 R_k^i \mu_k^i}(x_k^i)} h_k e^{u_k}$$
for all $R>0$ and $k$ large (depending on $R$).
\end{proof}

%%%%%%%%%%%%%%%%%%%%%%%%%%%%
\section{Appendix} \label{Appendix}
\noindent 
%%%%%%%%%%%%%%%%%%%%%%%%%%%%
For the sake of completeness, we give below the proof of Theorem \ref{352}.
\begin{proof}
By the Concentration-Compactness Principle and \eqref{824} we know that 
\begin{equation} \label{0930}
\max_{\overline{B_{2 \delta}(0)}}u_k \to +\infty,\qquad u_k \to -\infty \hbox{ locally uniformly in }\overline{B_{2\delta}(0)} \setminus \{0\}.
\end{equation}
Let $x_k=x_k^1$ be the sequence of maximum points of $u_k$ in $\overline{B_{2\delta}(0)}$: $u_k(x_k)=\displaystyle \max_{\overline{B_{2\delta}(0)}}u_k$. If \eqref{341} does already hold, the result is established by simply taking $k=1$ and $\mu_k=\mu_k^1$ according to \eqref{828}, since \eqref{339} follows by \eqref{0930} and the proof of \eqref{340} is classical and indipendent on the validity of \eqref{341}. Indeed, $U_k(y)=u_k(\mu_k y+x_k)+n \log \mu_k$ satisfies $U_k(y)\leq U_k(0)=0$ in $B_{\frac{2\delta}{\mu_k}}(0)$ and 
\begin{equation} \label{1040}
-\Delta_n U_k=h_k(\mu_k y+x_k) e^{U_k} \hbox{ in }\frac{\Omega -x_k}{\mu_k},\qquad \int_{\frac{\Omega -x_k}{\mu_k} }e^{U_k}= \int_{\Omega} e^{u_k}.
\end{equation}
Since $\frac{\Omega -x_k}{\mu_k} \to \mathbb{R}^n$ as $k \to +\infty$ in view of \eqref{339} and $B_{3\delta}(0) \subset \Omega$, by \eqref{758bis}, \eqref{eq:Peso1} and elliptic estimates \cite{Dib,Tol} we deduce that, up to a subsequence, $U_k \to U$ in $C^1_{loc}(\mathbb{R}^n)$, where $U$ solves \eqref{limitpb} with $p=0$. Notice that $h(0)=0$ would imply that $U$ is an upper-bounded $n-$harmonic function in $\mathbb{R}^n$ and therefore a constant function (see for instance Corollary 6.11 in \cite{HKM}), contradicting $\int_{\mathbb{R}^n} e^U<\infty$. As a consequence, we deduce that $h(0)>0$ and $U$ is the unique solution of \eqref{limitpb} given by \eqref{eq:Bubble} with $p=0$.

\medskip \noindent Assume that \eqref{341} does not hold with $x_k=x_k^1$ and proceed by induction. Suppose to have found $x_k^1,\dots,x_k^l$ so that \eqref{339}-\eqref{340} and \eqref{828} do hold. If \eqref{341} is not valid for $x_k^1,\dots,x_k^l$, in view of \eqref{0930} we construct $\bar x_k \in B_{2\delta}(0)$ as
\begin{equation} \label{0957} 
u_k(\bar x_k)+n \log \min_{i=1,\dots,l} |\bar x_k-x_k^i| =\max_{\overline{B_{2\delta}(0)}} [u_k+n \log \min_{i=1,\dots,l} |x-x_k^i|] \to +\infty
\end{equation}
and have that \eqref{339} is still valid for $x_k^1,\dots,x_k^l,\bar x_k$ with $\bar \mu_k=e^{-\frac{u_k(\bar x_k)}{n}}$ as it follows  by \eqref{340} for $i=1,\dots,l$ and \eqref{0957}. 

\medskip
Let us argue in a similar way as in the proof of Theorem \ref{711}. Observe that
$\displaystyle \min_{i=1,\dots,l} |\bar x_k+\bar \mu_k y-x_k^i| \geq \frac{1}{2} \displaystyle  \min_{i=1,\dots,l} |\bar x_k-x_k^i|$ and
$$u_k(\bar \mu_k y+\bar x_k) +n \log \bar \mu_k \leq n \log \min_{i=1,\dots,l} |\bar x_k-x_k^i|-n \log \min_{i=1,\dots,l} |\bar \mu_k y+\bar x_k-x_k^i| \leq n \log 2$$
for $|y|\leq R_k=\frac{1}{2 \bar \mu_k} \displaystyle \min_{i=1,\dots,l} |\bar x_k-x_k^i|$ in view of \eqref{0957}. Hence $\bar U_k(y)=u_k(\bar \mu_k y+\bar x_k) +n \log \bar \mu_k $ satisfies the analogue of \eqref{1040} with $\bar U_k \leq n \log 2$ in $B_{R_k}(0)$. Since $R_k \to +\infty$ in view of \eqref{339} for $x_k^1,\dots,x_k^l,\bar x_k$, up to a subsequence, by elliptic estimates \cite{Dib,Tol} $\bar U_k \to \bar U$ in $C^1_{loc}(\mathbb{R}^n)$, where $\bar U$ is a solution of \eqref{Limitingn} with $p=0$. By the classification result \cite{Esp} we know that $\bar U=U_{a,\lambda}$ for some $(a,\lambda) \in \mathbb R^n \times (0,\infty)$. Since $\bar U$ is a radially strictly decreasing function with respect to $a$, there exists a sequence $a_k \to a$ as $k \to +\infty$ so that
\begin{equation}\label{1058}
\bar U_k(a_k)=\max_{B_R(a_k)}\bar U_k
\end{equation}
for all $R>0$ and $k$ large (depending on $R$). Setting $x_k^{l+1}=\bar \mu_k a_k+\bar x_k$, since 
$\mu_k^{l+1}=e^{-\frac{u_k(x_k^{l+1})}{n}}$ satisfies
\begin{equation} \label{1152}
\frac{\bar \mu_k}{\mu_k^{l+1}}=e^{\frac{\bar U_k(a_k)}{n}} \to
e^{\frac{\max_{\mathbb{R}^n}  \bar U}{n}}
\end{equation}
as $k \to +\infty$, we deduce that  \eqref{339} is valid for $x_k^1,\dots,x_k^{l+1}$ and \eqref{828} follows by \eqref{1058} with some $R> \displaystyle e^{- \frac{\max_{\mathbb{R}^n}  \bar U}{n}}$. Since $U_k^{l+1}=u_k(\mu_k^{l+1}y+x_k^{l+1})+n \log \mu_k^{l+1}$ satisfies $U_k^{l+1}(y)\leq U_k^{l+1}(0)=0$ in $B_{R \frac{\bar \mu_k}{\mu_k^{l+1}}}(0)$ in view of \eqref{1058}, by 
\eqref{758bis}, \eqref{eq:Peso1}, \eqref{1152} and elliptic estimates \cite{Dib,Tol} we deduce that, up to a subsequence, $U_k^{l+1} \to U$ in $C^1_{loc}(\mathbb{R}^n)$, where $U$ is the unique solution of \eqref{limitpb} given by \eqref{eq:Bubble} with $p=0$, establishing the validity of \eqref{340} for $i=l+1$ too.\\

\medskip \noindent Since \eqref{339}-\eqref{340} and \eqref{828} on $x_k^1,\dots,x_k^l$ imply
\begin{eqnarray*}
\lim_{k \to +\infty} \int_{B_{3\delta}(0)} h_k e^{u_k} \geq \lim_{R \to +\infty} \lim_{k \to +\infty} \sum_{i=1}^l \int_{B_{R \mu_k^i}(x_k^i)} h_k e^{u_k}
= l c_n \omega_n
\end{eqnarray*}
thanks to \eqref{eq:Peso1}, \eqref{quantization} and \eqref{340}, in view of \eqref{824} the inductive process must stop after a finite number of iterations, say $N$, yielding the validity of Theorem \ref{352} with $x_k^1,\dots,x_k^N$.
\end{proof}

\bibliographystyle{plain}

\end{document}